\documentclass{amsart}
\usepackage{latexsym}
\usepackage{amstext}
\usepackage{amsmath}
\usepackage{amssymb}
\usepackage{amsthm}
\usepackage{mathrsfs} 
\usepackage{url}
\usepackage{comment}

\theoremstyle{plain}
\newtheorem{theorem}{Theorem}[section]

\newtheorem{corollary}[theorem]{Corollary}
\newtheorem{def-thm}[theorem]{Definition-Theorem}
\newtheorem{lemma}[theorem]{Lemma}

\theoremstyle{definition}
\newtheorem{definition}[theorem]{Definition}

\newtheorem{remark}[theorem]{Remark}
\newtheorem{example}[theorem]{Example}

\def \O {\mathcal{O}}
\def \II {\mathcal{I}}

\def \QQ {\mathbb{Q}}

\def \PP {\mathbb{P}}

\DeclareMathOperator{\Supp}{Supp}

\DeclareMathOperator{\codim}{codim}

\allowdisplaybreaks
\begin{document}

\title[Schmidt-Nochka Theorem for closed subschemes]{A Schmidt-Nochka Theorem for closed subschemes in subgeneral position}

\begin{abstract}
In previous work, the authors established a generalized version of Schmidt's subspace theorem for closed subschemes in general position in terms of Seshadri constants. We extend our theorem to weighted sums involving closed subschemes in subgeneral position, providing a joint generalization of Schmidt's theorem with seminal inequalities of Nochka. A key aspect of the proof is the use of a lower bound for Seshadri constants of intersections from algebraic geometry, as well as a generalized Chebyshev inequality. As an application, we extend inequalities of Nochka and Ru-Wong from hyperplanes in $m$-subgeneral position to hypersurfaces in $m$-subgeneral position in projective space, proving a sharp result in dimensions $2$ and $3$, and coming within a factor of $3/2$ of a sharp inequality in all dimensions. We state analogous results in Nevanlinna theory generalizing the Second Main Theorem and Nochka's theorem (Cartan's conjecture).

\end{abstract}

\author{Gordon Heier}\author{Aaron Levin}
\address{Department of Mathematics\\University of Houston\\4800 Calhoun Road\\ Houston, TX 77204\\USA}
\email{heier@math.uh.edu}

\address{Department of Mathematics\\Michigan State University\\619 Red Cedar Road\\ East Lansing, MI 48824\\USA}
\email{adlevin@math.msu.edu}

\thanks{The first author was supported by a grant from the Simons Foundation (Grant Number 963755-GH)}
\thanks{The second author was supported in part by NSF grants DMS-2001205 and DMS-2302298}



\maketitle
\section{Introduction}
Schmidt's subspace theorem occupies a central position in the theory of higher-dimensional Diophantine approximation. It has been generalized by many authors, including Evertse and Ferretti (see \cite{ef_festschrift}), Corvaja and Zannier (see \cite{CZ}), and Ru and Vojta \cite{RV}. In our recent paper \cite{HL21}, we proved the following version in terms of closed subschemes in general position and their Seshadri constants with respect to a given ample divisor $A$ (see also recent work of Ru and Wang \cite{RW22} and Vojta \cite{vojta_subscheme} involving beta constants in place of Seshadri constants).
\begin{theorem}[{\cite[Theorem 1.3]{HL21}}]\label{ajm_main}
\label{AJM_mthm}
Let $X$ be a projective variety of dimension $n$ defined over a number field $k$.  Let $S$ be a finite set of places of $k$.  For each $v\in S$, let $Y_{0,v},\ldots, Y_{n,v}$ be closed subschemes of $X$, defined over $k$, and in general position.  Let $A$ be an ample Cartier divisor on $X$, and $\epsilon>0$.  Then there exists a proper Zariski-closed subset $Z\subset X$ such that for all points $P\in X(k)\setminus Z$,
\begin{equation*}
\sum_{v\in S}\sum_{i=0}^n \epsilon_{Y_{i,v}}(A) \lambda_{Y_{i,v},v}(P)< (n+1+\epsilon)h_A(P).
\end{equation*}
\end{theorem}

In the paper \cite{HL20}, we investigated the implications of this theorem for the degeneracy of integral points and entire curves in the complement of nef effective divisors, culminating in new (quasi-)hyperbolicity theorems under weak positivity assumptions based on seeking a certain lexicographical minimax.

In the present paper, we extend Theorem \ref{AJM_mthm} to the following Schmidt-Nochka-type Theorem for weighted sums involving arbitrary closed subschemes (i.e., without any general position assumption). 

\begin{theorem}
\label{main}
Let $X$ be a  projective variety of dimension $n$  defined over a number field $k$, and let $S$ be a finite set of places of $k$. For each $v\in S$, let $Y_{1,v},\ldots, Y_{q,v}$ be closed subschemes of $X$, defined over $k$ (not necessarily in general position), and let $c_{1,v},\ldots, c_{q,v}$ be nonnegative real numbers.  For a closed subset $W\subset X$ and $v\in S$, let
\begin{align*}
\alpha_v(W)=\sum_{\substack{i\\ W\subset \Supp Y_{i,v}}}c_{i,v}.
\end{align*}
Let $A$ be an ample Cartier divisor on $X$, and $\epsilon>0$.  Then there exists a proper Zariski closed subset $Z$ of $X$ such that
\begin{align*}
\sum_{v\in S}\sum_{i=1}^q c_{i,v}\epsilon_{Y_{i,v}}(A) \lambda_{Y_{i,v},v}(P)<  \left((n+1)\max_{\substack{v\in S\\ \emptyset\subsetneq W\subsetneq X}}\frac{\alpha_v(W)}{\codim W}+\epsilon\right)h_A(P)
\end{align*}
for all points $P\in X(k)\setminus Z$.
\end{theorem}
Our flexible use of weights in the above inequality appears to be completely novel and will be crucial in Section \ref{Nochkalow}. Furthermore, Theorem \ref{main} can be considered sharp by the following example.

\begin{example}\label{sharp}
Fix a point $P_0\in \mathbb{P}^n(k)$ and let $r\geq 1$ be an integer. Consider hyperplanes $H_1,\ldots, H_{rn}$ (not depending on $v\in S$) which intersect at $P_0$ but otherwise intersect generally. Let $H$ be a hyperplane intersecting $H_1,\ldots, H_{rn}$ generally and set $H_{rn+i}=H$, $i=1,\ldots, r$, and $c_{i,v}=1$ for all $i$ and $v\in S$. Then it is easy to verify that the hyperplanes are in $m$-subgeneral position with $m=rn$ and
\begin{align*}
\max_{\substack{v\in S\\ \emptyset\subsetneq W\subsetneq X}}\frac{\alpha_v(W)}{\codim W}=\frac{m}{n}=\frac{rn}{n}=r,
\end{align*}
where the maximum is attained at $W=P_0$.
Therefore, Theorem \ref{main}, with the ample divisor $A$ being any hyperplane and all Seshadri constants equal to $1$, yields the trivial bound 
\begin{align*}
\sum_{i=1}^{r(n+1)} m_{H_i,S}(P)< (r(n+1)+\epsilon)h(P)
\end{align*}
for all points $P\in \mathbb{P}^n(k)\setminus Z$. Note that any line $L$ passing through $P_0$  and not contained in any of the hyperplanes $H_1,\ldots,H_{rn}$ intersects  $H_1,\ldots, H_{r(n+1)}$ in at most two distinct points. Since such a line $L$ contains an infinite set of $(\sum_{i=1}^{r(n+1)}H_i,S)$-integral points (and such lines are Zariski dense in $\mathbb{P}^n$), we cannot replace $r(n+1)$ in the inequality above by anything smaller. This gives a family of examples where Theorem \ref{main} is sharp.
\end{example}

We now explain the connections with the classical work of Schmidt and Nochka. To begin, Theorem \ref{main} recovers a weighted closed subscheme version of many recent inequalities generalizing Schmidt's Subspace Theorem, including work of Quang \cite{Quang, Quang_Pacific, Quang22}, Ji-Yan-Yu \cite{JYY}, and Shi \cite{Shi}.  In particular, if $Y_{1,v},\ldots, Y_{q,v}$ are in general position and $c_{i,v}=1$ for all $i$ and $v$, we have by definition that $\codim W\geq \alpha_v(W)$ for all $v\in S$ and $W\neq \emptyset$, and we recover Theorem~\ref{AJM_mthm}. We discuss in more detail the implications and relations with this previous work in Section \ref{quang_section}.

On the other hand, to justify the reference to Nochka, we recall that in 1982-1983, Nochka \cite{Noc82a,Noc82b,Noc83} proved  a conjecture of Cartan on defects of holomorphic curves in $\PP^n$ relative to a possibly degenerate set of hyperplanes. Nochka's work was further explained and simplified by Chen \cite{Chen} and Ru and Wong \cite{RW91}, with the latter work proving versions of Nochka's results in Diophantine approximation, including the following inequality for proximity functions associated to hyperplanes in $m$-subgeneral position:  

\begin{theorem}[Ru-Wong \cite{RW91}]
\label{tRuWong}
Let $H_{1},\ldots, H_{q}$ be hyperplanes of $\mathbb{P}^n$, defined over a number field $k$, and in $m$-subgeneral position. Let $S$ be a finite set of places of $k$ and let $\epsilon>0$. Then there exists a finite union of hyperplanes $Z\subset\mathbb{P}^n$ such that
\begin{align}
\label{RWineq}
\sum_{i=1}^q m_{H_{i},S}(P)< (2m-n+1+\epsilon)h(P)
\end{align}
for all points $P\in \mathbb{P}^n(k)\setminus Z$.
\end{theorem}

It is known that the quantity $2m-n+1$ in the inequality is sharp for all values of $m\geq n$ (for an appropriate configuration of hyperplanes).  A key inequality in Ru and Wong's work is the following inequality involving a sum of local heights weighted by Nochka weights (the analogous inequality in Nevanlinna theory plays a similar key role in Nochka's proof of Cartan's conjecture):
\begin{theorem}[Ru-Wong \cite{RW91}]
\label{RWNochka}
Let $S$ be a finite set of places of a number field $k$, and for each $v\in S$, let $H_{1,v},\ldots, H_{q,v}$ be hyperplanes of $\mathbb{P}^n$, defined over $k$ and in $m$-subgeneral position, with associated Nochka weights $\omega_{1,v},\ldots, \omega_{q,v}$. Let $\epsilon>0$. Then there exists a finite union of hyperplanes $Z$ of $\mathbb{P}^n$ such that
\begin{align*}
\sum_{v\in S}\sum_{i=1}^q \omega_{i,v}\lambda_{H_{i,v},v}(P)< (n+1+\epsilon)h(P)
\end{align*}
for all points $P\in \mathbb{P}^n(k)\setminus Z$.
\end{theorem}

A fundamental property of Nochka weights is that for a nonempty closed subset $W\subset X$ and $v\in S$,
\begin{align*}
\sum_{\substack{i\\ W\subset H_{i,v}}}\omega_{i,v}\leq \codim W.
\end{align*}
Thus, we see that Theorem \ref{RWNochka} is a consequence of Theorem \ref{main} (with $c_{i,v}=\omega_{i,v}$ and $Y_{i,v}=H_{i,v}$), excluding the linearity of the exceptional set $Z$. Then we may view Theorem \ref{main} as jointly generalizing Schmidt's Subspace Theorem as well as inequalities arising out of Nochka's work and its extensions.

Just as Theorem \ref{RWNochka} is a key ingredient in proving the inequality \eqref{RWineq}, we derive from Theorem \ref{main} a version of Ru-Wong's inequality \eqref{RWineq} for arbitrary effective divisors, under a Bezout-type intersection assumption which is, in particular, valid for projective space. 

\begin{theorem}\label{nochka_type_theorem}
Let $X$ be a  projective variety of dimension $n$  defined over a number field $k$ and let $S$ be a finite set of places of $k$.  Let $D_1,\ldots, D_q$ be effective Cartier divisors on $X$, defined over $k$, in $m$-subgeneral position.  Let $D_I=\cap_{i\in I}D_i$.  We assume the following Bezout property holds for intersections among the divisors: If $I,J\subset \{1,\ldots, q\}$ then
\begin{align*}
\codim D_{I\cup J}=\codim (D_I\cap D_J)\leq \codim D_I+\codim D_J.
\end{align*}
Let $A$ be an ample divisor on $X$ and let $\epsilon>0$.  Then there exists a proper Zariski closed subset $Z$ of $X$ such that
\begin{align}
\label{almostNochka}
\sum_{i=1}^q \epsilon_{D_{i}}(A) m_{D_{i},S}(P)< \left(\frac{3}{2}(2m-n+1)+\epsilon\right)h_A(P)
\end{align}
for all $P\in X(k) \setminus Z$.
\end{theorem}

As is well-known, the Bezout property holds when $X=\mathbb{P}^n$, and if $D_1,\ldots, D_q$ are hypersurfaces in $\mathbb{P}^n$ of degrees $d_1,\ldots, d_q\geq 1$, and $h_A=h$ is the standard height (associated to $A$ a hyperplane), we have $\epsilon_{D_i}(A)=\frac{1}{d_i}$ and the inequality \eqref{almostNochka} becomes (outside some proper closed subset of $\mathbb{P}^n$):
\begin{align*}
\sum_{i=1}^q \frac{1}{d_i}m_{D_{i},S}(P)< \left(\frac{3}{2}(2m-n+1)+\epsilon\right)h(P).
\end{align*}
Thus, we obtain a version of (Nochka-)Ru-Wong's Theorem \ref{tRuWong} with an extra factor of $\frac{3}{2}$, but valid for arbitrary hypersurfaces in projective space. This inequality does not seem to have been previously known (for all $m$ and $n$) even with $\frac{3}{2}$ replaced by an arbitrarily large constant; an inequality with a factor of $(m-n+1)(n+1)+\epsilon$ on the right-hand side follows from work of Quang \cite{Quang}.

In low dimensions ($n\leq 3$) we are able to remove the factor $\frac{3}{2}$ and prove a full generalization of Ru-Wong's theorem to hypersurfaces in projective space.

\begin{theorem}
\label{nochka_type_theorem2}
Let $n\leq 3$ be a positive integer.  Let $D_1,\ldots, D_q$ be effective divisors on $\mathbb{P}^n$, defined over $k$, in $m$-subgeneral position, of degrees $d_1,\ldots, d_q$.  Let $\epsilon>0$ and let $S$ be a finite set of places of $k$. Then there exists a proper Zariski closed subset $Z$ of $\mathbb{P}^n$ such that
\begin{align*}
\sum_{i=1}^q \frac{1}{d_i} m_{D_{i},S}(P)<(2m-n+1+\epsilon)h(P)
\end{align*}
for all $P\in \mathbb{P}^n(k) \setminus Z$.
\end{theorem}

In fact, we prove more general inequalities in dimension at most $3$ under suitable geometric assumptions, and more generally for weighted sums of proximity functions (Theorems \ref{nochka23irred} and \ref{theorempic1}).

Using the well-known correspondence between statements in Diophantine approximation and Nevanlinna theory \cite{Vojta_LNM}, the proof of Theorem \ref{main} can be adapted to prove the following generalization of the Second Main Theorem in Nevanlinna theory:

\begin{theorem}
\label{mthmNev}
Let $X$ be a complex projective variety of dimension $n$. Let $Y_1,\ldots, Y_q$ be closed subschemes of $X$  and let $c_{1},\ldots, c_q$ be nonnegative real numbers.  Let $f:\mathbb{C}\to X$ be a holomorphic map with Zariski dense image, $A$ an ample Cartier divisor on $X$, and $\epsilon>0$.  Then 
\begin{equation*}
\int_{0}^{2\pi}\max_J \sum_{j\in J} c_j\epsilon_{Y_j}(A) \lambda_{Y_j}(f(re^{i\theta}))\frac{d\theta}{2\pi}\leq_{\operatorname{exc}} (\Delta(n+1)+\epsilon)T_{f,A}(r),
\end{equation*}
where the maximum is taken over all subsets $J$ of $\{1,\dots, q\}$ such that for every nonempty closed subset $W\subset X$,
\begin{align*}
\sum_{\substack{j\in J\\ W\subset \Supp Y_{j}}}c_{j}\leq \Delta \codim W.
\end{align*}
\end{theorem}
The notation $\leq_{\operatorname{exc}}$ means that the inequality holds for all $r\in (0,\infty)$ outside of a set of finite Lebesgue measure. When $c_i=1$ for all $i$, this was proven (with a slightly different $\Delta$) independently by Quang \cite{Quang22} (see the remarks on the analogous Diophantine result after Theorem \ref{Quangmain}).

Analogous to Theorems \ref{nochka_type_theorem} and \ref{nochka_type_theorem2} and their proofs, we have the following consequences for holomorphic curves and proximity functions associated to divisors in $m$-subgeneral position. In particular, we extend Nochka's theorem (Cartan's conjecture) to hypersurfaces in projective space, however with an extra factor of $\frac{3}{2}$ on the right-hand side of the inequality (and we eliminate this extra factor in dimensions $\leq 3$).

\begin{theorem}\label{nochka_type_theorem_complex}
Let $X$ be a complex projective variety of dimension $n$.  Let $D_1,\ldots, D_q$ be effective Cartier divisors on $X$ in $m$-subgeneral position.  let $D_I=\cap_{i\in I}D_i$.  Assume the following Bezout property holds for intersections among the divisors: If $I,J\subset \{1,\ldots, q\}$ then
\begin{align*}
\codim D_{I\cup J}=\codim (D_I\cap D_J)\leq \codim D_I+\codim D_J.
\end{align*}
Let $f:\mathbb{C}\to X$ be a holomorphic map with Zariski dense image, $A$ an ample Cartier divisor on $X$, and $\epsilon>0$.  Then 
\begin{align*}
\sum_{i=1}^q \epsilon_{D_{i}}(A) m_{f,D_i}(r)\leq_{\operatorname{exc}} \left(\frac{3}{2}(2m-n+1)+\epsilon\right)T_{f,A}(r)(P).
\end{align*}
\end{theorem}

In small dimensions we have the better result:

\begin{theorem}
\label{nochka_type_theorem2_complex}
Let $n\leq 3$ be a positive integer.  Let $D_1,\ldots, D_q$ be effective divisors on $\mathbb{P}^n$ in $m$-subgeneral position of degrees $d_1,\ldots, d_q$.  Let $f:\mathbb{C}\to \mathbb{P}^n$ be a holomorphic map with Zariski dense image, and let $\epsilon>0$. Then 
\begin{align*}
\sum_{i=1}^q \frac{1}{d_i} m_{f,D_{i}}(r)\leq_{\operatorname{exc}}(2m-n+1+\epsilon)T_f(r).
\end{align*}
\end{theorem}

The proof of Theorem \ref{main} is based on three main ingredients: a lower bound for Seshadri constants of intersections as described in \cite[Section 5.4]{PAGI}, an inequality that may be viewed as a generalization of Chebyshev's inequality (Lemma \ref{elemineq}), together with an application of Theorem \ref{AJM_mthm} to certain auxiliary closed subschemes obtained as intersections.  The applicability of our generalized Chebyshev inequality, for which we provide a complete proof due to the lack of a reference, stems from the fact that our definition of $m$-subgeneral position of closed subschemes, even when specialized to the case of general position, allows a closed subscheme $Y$ of codimension $r$ to be repeated $r$ times while general position is maintained. Due to this, it is not clear how to extend our results to similar inequalities involving beta constants in place of Seshadri constants, as the former inequalities use a somewhat more restrictive notion of general position.\par 

The structure of this paper is as follows. In Section \ref{sec_def_prelim}, we recall basic definitions and give a standard lemma with a short proof based on our general philosophy. Subsequently, we prove Theorem \ref{main} in Section \ref{section_main_thm_proof}. In Section \ref{quang_section}, we discuss the relation of Theorem \ref{main} with results of Quang and others on Diophantine approximation to divisors in subgeneral position. Finally, in Sections \ref{nochka_section} and \ref{Nochkalow} we prove the inequalities of Nochka-type.

\section{Definitions and preliminaries}\label{sec_def_prelim}
We begin with a brief account of the key properties of local and global height functions, following mostly the reference \cite{silverman_87}. We will then introduce the notion of subgeneral position used in this paper. Finally, we revisit a standard lemma (Lemma \ref{lemma}) by giving a very simple proof of it based on the general philosophy developed in our work.
\subsection{Local height functions}\label{lhf}
Let $Y$ be a closed subscheme of a projective variety $X$, both defined over a number field $k$.  For any place $v$ of $k$, one can associate a {\it local height function} (or {\it Weil function}) $\lambda_{Y,v}:X(k)\setminus Y\to \mathbb{R}$, well-defined up to $O(1)$, which gives a measure of the $v$-adic distance of a point to $Y$, being large when the point is close to $Y$. If $Y=D$ is an effective (Cartier) divisor (which we will frequently identify with the associated closed subscheme), these height functions agree with the usual height functions associated to divisors.  Local height functions satisfy the following properties: if $Y$ and $Z$ are two closed subschemes of $X$, defined over $k$, and $v$ is a place of $k$, then up to $O(1)$,
\begin{align*}
\lambda_{Y\cap Z,v}&=\min\{ \lambda_{Y,v},\lambda_{Z,v}\},\\
\lambda_{Y+Z,v}&=\lambda_{Y,v}+\lambda_{Z,v},\\
\lambda_{Y,v}&\leq \lambda_{Z,v}, \text{ if }Y\subset Z.
\end{align*}
In particular, $\lambda_{Y,v}$ is bounded from below for all $P\in X(k)\setminus Y$.
  If $\phi:W\to X$ is a morphism of projective varieties with $\phi(W) \not \subset Y$, then up to $O(1)$,
\begin{equation*}
\lambda_{Y,v}(\phi(P))=\lambda_{\phi^*Y,v}(P), \quad \forall P\in W(k)\setminus \phi^*Y.
\end{equation*}
Here, $Y\cap Z$, $Y+Z$, $Y\subset Z$, and $\phi^*Y$ are defined in terms of the associated ideal sheaves.  In particular, we emphasize that if $Y$ corresponds to the ideal sheaf $\II_Y$, then $\phi^*Y$ is the closed subscheme corresponding to the inverse image ideal sheaf $\phi^{-1}\II_Y\cdot \mathcal{O}_W$.  \par
For our purposes, it is important to note that one can set $Y=Z$ and obtain
$$\lambda_{2Y,v}= \lambda_{Y+Y,v}=\lambda_{Y,v}+\lambda_{Y,v}=2\lambda_{Y,v},$$
where $2Y$ is the closed subscheme corresponding to the ideal sheaf $\II_Y^2$. Moreover, it holds that if $Y_1,\ldots,Y_m$ are closed subschemes of $X$ and $c_1,\ldots,c_m$ are positive integers, then
\begin{align}
\lambda_{c_1Y_1\cap\ldots\cap c_mY_m,v}&=\min\{ c_1\lambda_{Y_1,v},\ldots,c_m\lambda_{Y_m,v}\},\label{cap}\\
\lambda_{c_1Y_1+\cdots+c_mY_m,v}&=c_1\lambda_{Y_1,v}+\cdots+c_m\lambda_{Y_m,v}.\label{plus}
\end{align}
\subsection{Global height functions} If $Y$ is a closed subscheme of a projective variety $X$, both defined over a number field $k$, then a {\it global height function} $h_Y:X(k)\setminus Y\to \mathbb{R}$, can be associated as follows. Letting $M_k$ denote the set of all places of $k$, and set 
$$h_Y(P) = \sum_{v\in M_k} \lambda_{Y,v} (P).$$
A subtle point here is that the definition of the $\lambda_{Y,v}$, $v\in M_k$, can be made so that $(\lambda_{Y,v})_{v\in M_k}$ is well-defined up to a so-called $M_k$-constant, which is a collection of real numbers  $(c_v)_{v\in M_K}$ such that $c_v = 0$ for all but finitely many $v$.  Therefore, $h_Y$ is well-defined up to $O(1)$. \par
For a given finite set $S$ of places of $k$, it is customary to split the above sum into the {\it proximity function}
$$m_{Y,S}(P) =\sum_{v\in S} \lambda_{Y,v} (P)$$
and the {\it counting function}
$$N_{Y,S}(P) =\sum_{v\in M_k \setminus S} \lambda_{Y,v} (P).$$

Global height functions satisfy properties analogous to those stated above for local height functions, except the first property above, which becomes $h_{Y\cap Z}\leq\min\{ h_{Y},h_{Z}\}+O(1)$. When not addressed explicitly and even when any mention of $M_k$-constants or $O(1)$ is omitted for brevity's sake, all equalities and inequalities involving height functions are to be interpreted with their mild ill-definedness in mind, which is standard procedure.\par
A special role is played by global height functions of Cartier divisors, in particular, of ample Cartier divisors. Namely, if $A$ is an ample Cartier divisor, the global height function $h_A$ measures the arithmetic complexity of the point $P$, explaining the presence of $h_A$ on the right hand side of the inequalities in Theorems \ref{ajm_main} and \ref{main}. Furthermore, for an arbitrary closed subscheme $Y$, the height function $h_A$ dominates $h_Y$ in the sense that the domain of definition of $h_A$ can be extended to all of $X(k)$ and there exists a constant $c$ such that 
\begin{align} 
h_Y(P)\leq c h_A(P) \label{ample_domination}
\end{align}
for all $P\in X(k)\setminus Y$. For a proof of this statement, see \cite[Proposition 4.2]{silverman_87} (or Lemmas \ref{lemma} and \ref{lemma2} for refined statements). As explained in this reference, the intuition for this result is that if $P$ is $v$-adically close to Y (but not on $Y$), then $P$ must be $v$-adically complicated in terms of its coordinates after a projective embedding by a very ample multiple of the divisor $A$. \par
A further property of the local height functions $(\lambda_{Y,v})_{v\in M_k}$, is that they are in fact bounded from below by an $M_k$-constant, which immediately yields that the inequality \eqref{ample_domination} can be stated for any individual $\lambda_{Y,v}$, i.e., 
\begin{align} 
\lambda_{Y,v}(P)\leq c h_A(P). \label{lhf_ample_domination}
\end{align}

\subsection{Subgeneral position and Seshadri constants}\label{ex_int_prelim}
We recall the definition of being in $m$-subgeneral position and the definition of Seshadri constants, both in the context of closed subschemes.
\begin{definition}\label{gen_pos_def}
If $X$ is a projective variety of dimension $n$, we say that closed subschemes $Y_1,\ldots, Y_q$ of $X$ are in {\it $m$-subgeneral position} if for every subset $I\subset\{1,\ldots, q\}$ with $|I|\leq m+1$ we have $\codim \cap_{i\in I} Y_i\geq |I|+n-m$, where we use the convention that $\dim \emptyset=-1$. In the case $m=n$, we say that the closed subschemes are in {\it general position}.\par
\end{definition}
It should be noted that even when $m=n$, a closed subscheme of codimension $r$ may be repeated $r$ times while maintaining general position according to our definition. This phenomenon will be important in the application of the generalized Chebyshev inequality in Section 5.

We use a weighted notion of $m$-subgeneral position in Section \ref{Nochkalow}.

We now recall the notion of a {\it Seshadri constant} for a closed subscheme relative to a nef Cartier divisor (see \cite[Section 2]{HL21} for further details) . 

\begin{definition}\label{Sesh}
Let $Y$ be a closed subscheme of a projective variety $X$ and let $\pi:\tilde{X}\to X$ be the blowing-up of $X$ along $Y$. Let $A$ be a nef Cartier divisor on $X$. We define the {\it Seshadri constant} $\epsilon_Y(A)$ of $Y$ with respect to $A$ to be the real number
\begin{align*}
\epsilon_Y(A)=\sup\{\gamma\in {\mathbb{Q}}^{\geq 0}\mid \pi^*A-\gamma E\text{ is $\mathbb{Q}$-nef}\},
\end{align*}
where $E$ is an effective Cartier divisor on $\tilde X$ whose associated invertible sheaf is the dual of $\pi^{-1}\II_Y \cdot \O_{\tilde X}$.
\end{definition}

\subsection{A standard lemma}
We will use the following lemma (well-known in the numerically equivalent case) and take this opportunity to provide a very short proof of the lemma  based on the philosophy expressed in our earlier work \cite{HL21}.
\begin{lemma}\label{lemma}
Let $X$ be a projective variety of dimension $n$ defined over a number field $k$.  Let $A,B$ be effective Cartier divisors defined over $k$, with $A$ ample. Let $\epsilon_B(A)$ be the associated Seshadri constant. For all $\epsilon >0$, there is a constant $c_\epsilon$ such that
\begin{align*}
\epsilon_B(A)h_B \leq (1+\epsilon) h_A+c_\epsilon. 
\end{align*}
In particular, if $A$ and $B$ are numerically equivalent ample Cartier divisors, then for all $\epsilon >0$, there is a constant $c_\epsilon$ such that
$$(1-\epsilon)h_A -c_\epsilon \leq h_B \leq (1+\epsilon) h_A+c_\epsilon.$$ 
\end{lemma}
\begin{proof}
Without loss of generality, we can restrict ourselves to $\epsilon$ such that $\frac{1+\epsilon}{\epsilon_B(A)}$ is rational.  Then $\frac{1+\epsilon}{\epsilon_B(A)}A-B$ is an ample $\QQ$-divisor. Let $N$ be a sufficiently divisible integer such that $N\left(\frac{1+\epsilon}{\epsilon_B(A)}A-B\right)$ is an ample integral divisor. By the boundedness of heights from below, we can infer
$$h_{N\left(\frac{1+\epsilon}{\epsilon_B(A)}A-B\right)} \geq -c'_\epsilon$$
and thus
$$N h_B \leq N\frac{1+\epsilon}{\epsilon_B(A)}h_A+c'_\epsilon,$$
which yields the inequality. In particular, if $A$ and $B$ are numerically equivalent ample divisors then $\epsilon_B(A)=\epsilon_A(B)=1$, and the result follows by applying the inequality twice, once with the roles of $A$ and $B$ switched.
\end{proof}

More generally, one easily extends this result to closed subschemes:
\begin{lemma}\label{lemma2}
Let $X$ be a projective variety of dimension $n$ defined over a number field $k$.  Let $A$ be an ample divisor on $X$ and $Y$ a closed subscheme of $X$, both defined over $k$. Let $\epsilon_Y(A)$ be the associated Seshadri constant. For all $\epsilon >0$, there is a constant $c_\epsilon$ such that
\begin{align*}
\epsilon_Y(A)h_Y(P) \leq (1+\epsilon) h_A(P)+c_\epsilon
\end{align*}
for all $P\in X(k)\setminus \Supp Y$.
\end{lemma}
\begin{proof}
Let $\pi:\tilde{X}\to X$ be the blowing-up of $X$ along $Y$, and let $E$ be the associated exceptional divisor as in Definition \ref{Sesh}. Then from the definitions, $\epsilon_Y(A)=\epsilon_E(\pi^*A)$. While $\pi^*A$ is not necessarily ample, it is known \cite[Ex.~II.7.14(b)]{Hartshorne} that $\pi^*A-\delta E$ is an ample $\mathbb{Q}$-divisor for all sufficiently small (rational) $\delta$. Then using functoriality of heights, essentially the same proof as in Lemma \ref{lemma} (applied to $\pi^*A$ and $E$) yields the inequality.   
\end{proof}

\section{Proof of Theorem \ref{main}}\label{section_main_thm_proof}

For the proof of Theorem \ref{main}, we need the following inequality, which is a generalization of Chebyshev's inequality (see also work of Jensen \cite[p.~245]{MPF}). For lack of a reference, we provide a detailed proof. Due to the appearance of the denominator on the right hand side of the inequality, some care has to be taken in the statement and proof to avoid division by zero, and we therefore break our arguments up into several cases.
\begin{lemma}
\label{elemineq}
Let $a_1\geq a_2\geq\cdots \geq a_n\geq 0$ and $b_1,\ldots, b_n,c_1,\ldots, c_n$ be nonnegative real numbers. Assume that that there exists at least one index $i$ such that $c_{i}\not = 0$, and let $i_0$ be the smallest such index. Then
\begin{align}
\label{generalChebyshev}
\sum_{i=1}^na_ib_i\geq \left(\min_{i_0\leq j\leq n} \frac{\sum_{i=1}^jb_i}{\sum_{i=1}^jc_i}\right)\sum_{i=1}^na_ic_i.
\end{align}
\end{lemma}

\begin{remark}
If additionally $b_1\geq b_2\geq\cdots \geq b_n\geq 0$ and $c_1=\cdots=c_n=1$, then the minimum on the right-hand side of \eqref{generalChebyshev} occurs at $j=n$ and one obtains Chebyshev's inequality:
\begin{align*}
\sum_{i=1}^na_ib_i\geq \frac{1}{n}\sum_{i=1}^na_i\sum_{i=1}^nb_n.
\end{align*}
Similarly, if one assumes the $c_i$ are all positive and $\frac{b_1}{c_1}\geq \cdots \geq\frac{b_n}{c_n}$, then the minimum on the right-hand side of \eqref{generalChebyshev} again occurs at $j=n$, and one obtains an inequality of Jensen \cite[p.~245]{MPF}.
\end{remark}

\begin{proof}
We prove the inequality by induction on $n$. If $n=1$, the inequality is trivial. Let $n>1$. We may assume that $b_1\neq 0$ (if $b_1=0$ and $c_1\neq 0$ then the minimum in the inequality occurs at $j=1$ and the inequality is trivial; if $b_1=c_1=0$ the result follows from induction). Let $j_0\geq i_0$ be an index such that $\frac{\sum_{i=1}^jb_i}{\sum_{i=1}^jc_i}$ is minimized at $j=j_0$. \par
Case I. We start with the case $j_0<n$. Then for $j>j_0$, we have due to the minimality of $j_0$:
\begin{align*}
0\geq \frac{\sum_{i=1}^{j_0}b_i}{\sum_{i=1}^{j_0}c_i}-\frac{\sum_{i=1}^{j}b_i}{\sum_{i=1}^{j}c_i}&=\frac{\sum_{i=1}^{j}c_i\sum_{i=1}^{j_0}b_i-\sum_{i=1}^{j_0}c_i\sum_{i=1}^{j}b_i}{\sum_{i=1}^{j_0}c_i\sum_{i=1}^{j}c_i}\\
&=\frac{\sum_{i=j_0+1}^{j}c_i\sum_{i=1}^{j_0}b_i-\sum_{i=1}^{j_0}c_i\sum_{i=j_0+1}^{j}b_i}{\sum_{i=1}^{j_0}c_i\sum_{i=1}^{j}c_i}.
\end{align*}
Considering the numerator, we find that for all $j>j_0$,
\begin{align}\label{numerator_neg}
\sum_{i=j_0+1}^{j}c_i\sum_{i=1}^{j_0}b_i-\sum_{i=1}^{j_0}c_i\sum_{i=j_0+1}^{j}b_i \leq 0.
\end{align}

We now write 
\begin{align*}
\sum_{i=1}^na_ib_i&=\sum_{i=1}^{j_0}a_ib_i+\sum_{i=j_0+1}^na_ib_i
\end{align*}
and further distinguish cases.\par
Case I.a. Suppose that $b_i=0$ for all $i > j_0$ or $c_i=0$ for all $i>j_0$.  If $b_i=0$ for all $i > j_0$, then setting $j=n$ in \eqref{numerator_neg}, we obtain
\begin{align*}
\sum_{i=j_0+1}^{n}c_i\sum_{i=1}^{j_0}b_i- 0 \leq 0,
\end{align*}
which implies that $c_{j_0+1}=\ldots=c_n =0$. Thus in any case we may assume that $c_i=0$ for all $i>j_0$. By induction and the fact that the minimum clearly remains unchanged after replacing $j_0$ by $n$, we can infer that
\begin{align*}
\sum_{i=1}^na_ib_i&\geq \sum_{i=1}^{j_0}a_ib_i \geq \left(\min_{i_0\leq j\leq j_0} \frac{\sum_{i=1}^jb_i}{\sum_{i=1}^jc_i}\right)\sum_{i=1}^{j_0}a_ic_i=\left(\min_{i_0\leq j\leq n} \frac{\sum_{i=1}^jb_i}{\sum_{i=1}^jc_i}\right)\sum_{i=1}^{n}a_ic_i,
\end{align*}
proving the lemma in this case.\par
Case I.b. There exists $i,i' > j_0$ such that $b_i \not = 0$ and $c_{i'}\neq 0$. Let $k_0$ be the smallest index $i>j_0$ with $b_i \not = 0$. By the minimality of $j_0$ it is clear that $c_{j_0+1}=\ldots=c_{k_0-1}=0$. Let $i_1$ be the smallest index $i$ with $k_0\leq i \leq n$ such that $c_i\not = 0$.  Note that
\begin{align}\label{break_up}
\sum_{i=1}^na_ib_i&=\sum_{i=1}^{j_0}a_ib_i+\sum_{i=j_0+1}^na_ib_i =\sum_{i=1}^{j_0}a_ib_i+\sum_{i=k_0}^{n}a_ib_i.
\end{align}
Applying induction twice to \eqref{break_up} yields

\begin{align}\label{induction_twice}
\sum_{i=1}^na_ib_i& \geq  \left(\min_{i_0\leq j\leq j_0} \frac{\sum_{i=1}^jb_i}{\sum_{i=1}^jc_i}\right)\sum_{i=1}^{j_0}a_ic_i + \left(\min_{i_1\leq j\leq n} \frac{\sum_{i=k_0}^jb_i}{\sum_{i=k_0}^jc_i}\right)\sum_{i=k_0}^{n}a_ic_i. 
\end{align}

For $j \geq i_1$, all the sums appearing in \eqref{numerator_neg} are positive and we can divide to obtain for all $j \geq i_1$:
\begin{align*}
\frac{\sum_{i=1}^{j_0}b_i}{\sum_{i=1}^{j_0}c_i}\leq \frac{\sum_{i=j_0+1}^{j}b_i}{\sum_{i=j_0+1}^{j}c_i}.
\end{align*}
Since $b_{j_0+1}=\ldots=b_{k_0-1}=0=c_{j_0+1}=\ldots=c_{k_0-1}$, we can extend the above inequality to obtain for all $j \geq i_1$: 
\begin{align*}
\frac{\sum_{i=1}^{j_0}b_i}{\sum_{i=1}^{j_0}c_i}\leq \frac{\sum_{i=j_0+1}^{j}b_i}{\sum_{i=j_0+1}^{j}c_i}= \frac{\sum_{i=k_0}^{j}b_i}{\sum_{i=k_0}^{j}c_i}.
\end{align*}

Together with \eqref{induction_twice}, this proves the lemma in this case.\par
Case II. If $j_0=n$, then we let
\begin{align*}
\alpha=\frac{\sum_{i=1}^{n}b_i}{\sum_{i=1}^{n}c_i}.
\end{align*}
It follows that for all $j\geq 1:$
\begin{align*}
\sum_{i=1}^jb_i\geq \alpha\sum_{i=1}^jc_i.
\end{align*}
Using that $a_1\geq a_2\geq \cdots \geq a_n\geq 0$, we can now estimate
\begin{align*}
\sum_{i=1}^na_ib_i & =a_1b_1+\sum_{i=2}^na_ib_i=\alpha a_1c_1+a_1(b_1-\alpha c_1)+a_2b_2+\sum_{i=3}^na_ib_i\\
&\geq \alpha a_1c_1+a_2(b_1+b_2-\alpha c_1)+\sum_{i=3}^na_ib_i\\
&\geq \alpha a_1c_1+\alpha a_2c_2+a_2(b_1+b_2-\alpha c_1-\alpha c_2)+\sum_{i=3}^na_ib_i\\
&\geq \alpha a_1c_1+\alpha a_2c_2+a_3(b_1+b_2+b_3-\alpha c_1-\alpha c_2)+\sum_{i=4}^na_ib_i\\
&\vdots\\
&\geq \alpha\sum_{i=1}^n a_ic_i +a_n\left(\sum_{i=1}^ n b_i-\alpha\sum_{i=1}^nc_i\right)\\
&\geq \alpha\sum_{i=1}^n a_ic_i, 
\end{align*}
which concludes the proof of the lemma.
\end{proof}

An immediate corollary of Lemma \ref{elemineq} is
\begin{corollary}
\label{cor_elemineq}
Let $a_1\geq a_2\geq\cdots \geq a_n\geq 0$ and $b_1,\ldots, b_n,c_1,\ldots, c_n$ be nonnegative real numbers. Assume that $b_1\not = 0$. Then
\begin{align*}
 \left(\max_{1\leq j\leq n} \frac{\sum_{i=1}^jc_i}{\sum_{i=1}^jb_i}\right) \sum_{i=1}^na_ib_i\geq \sum_{i=1}^na_ic_i.
\end{align*}
\end{corollary}

We now conduct the proof of Theorem \ref{main}.

\begin{proof}[Proof of Theorem \ref{main} ]
Let $v\in S$. In order to handle the technical difficulty that the Seshadri constants $\epsilon_{Y_{1,v}}(A),\ldots, \epsilon_{Y_{q,v}}(A)$ may be unrelated distinct positive real numbers, we ``normalize" them in a two step process. First, for every $i=1\ldots, q$, let $\tau_{1,v},\ldots,\tau_{q,v}$ be real numbers in the interval $[0,1)$ so that 
$$(1+\tau_{1,v})\epsilon_{Y_{1,v}}(A),\ldots, (1+\tau_{q,v})\epsilon_{Y_{q,v}}(A)$$
are rational numbers. We choose $\tau_{i,v}=0$ if $\epsilon_{Y_{i,v}}(A)$ is already a rational number. For a sufficiently divisible integer $c_v$ (independent of $i$) such that 
$$ c_v (1+\tau_{1,v})\epsilon_{Y_{1,v}}(A),\ldots, c_v(1+\tau_{q,v})\epsilon_{Y_{q,v}}(A)$$
are all integers, we set
$$\tilde  Y_{i,v} := c_v (1+\tau_{i,v})\epsilon_{Y_{i,v}}(A)Y_{i,v},$$
where the multiplication of the scheme $Y_{i,v}$ with the integer $c_v (1+\tau_{i,v})\epsilon_{Y_{i,v}}(A)$ is supposed to be understood in the sense of schemes.
Observe that
\begin{equation}
\epsilon_{\tilde  Y_{i,v}}(A)=\frac 1 {c_v(1+\tau_{i,v})} \label{modif_sesh}
\end{equation}
for $i=1,\ldots,q$, and therefore
\begin{equation*}
\frac 1 {c_v} \leq 2\epsilon_{\tilde  Y_{i,v}}(A)
\end{equation*}
for $i=1,\ldots,q$. Consequently,
\begin{equation*}
\frac 1 {c_v} \leq 2 \min_{i}  \epsilon_{\tilde  Y_{i,v}}(A).
\end{equation*}

Note that
\begin{align}
&\ |\epsilon_{\tilde  Y_{i,v}}(A)-\epsilon_{\tilde  Y_{i',v}}(A)| \nonumber \\
= &\  \left|\frac 1 {c_v(1+\tau_{i,v})} - \frac 1 {c_v(1+\tau_{i',v})}\right| \nonumber \\
= &\  \frac 1 {c_v} \left|\frac 1 {(1+\tau_{i,v})} - \frac 1 {(1+\tau_{i',v})}\right| \nonumber \\
\leq  &\  \frac 1 {c_v} \left| \frac {\tau_{i',v}-\tau_{i,v}} {(1+\tau_{i',v})(1+\tau_{i,v})} \right|\nonumber \\
\leq  &\  \max_k\frac {\tau_{k,v}}  {c_v} \nonumber \\
\leq  &\   \max_k\tau_{k,v} \cdot 2  \min_{l}  \epsilon_{\tilde  Y_{l,v}}(A) \label{sesh_bound}
\end{align}
is satisfied for all $i,i'=1,\ldots,q$. 

We introduce the following notation for intersections of the normalized subschemes:
 $$\tilde Y_{I,v}=\bigcap_{i\in I}\tilde  Y_{i,v}.$$ 

Fix a point $P$ outside of the support of the given closed subschemes. Let $\{i_{1,v},\ldots, i_{q,v}\}=\{1,\ldots, q\}$ be such that (choosing, as we may, our local height functions to be nonnegative):
\begin{align}
\lambda_{\tilde Y_{i_{1,v},v},v}(P)\geq \ldots\geq \cdots \geq \lambda_{\tilde Y_{i_{q,v},v},v}(P)\geq 0. \label{tilde_chain_ineq}
\end{align}
We set $I_{j,v}=\{i_{1,v},\ldots, i_{j,v}\}$.  Let $m_v$ be the largest index such that
\begin{align*}
\cap_{j=1}^{m_v} \tilde Y_{i_{j,v},v}\neq \emptyset.
\end{align*}
Due to the chain of inequalities \eqref{tilde_chain_ineq} and property \eqref{cap} we have
\begin{align*}
\lambda_{\tilde Y_{i_{j,v}v},v}(P)=\min_{j'\leq j}\lambda_{\tilde Y_{i_{j',v},v},v}(P)=\lambda_{\tilde Y_{I_{j,v},v},v}(P).
\end{align*}
In particular, for $j>m_v$, $\lambda_{\tilde Y_{i_{j,v},v},v}(P)$ is bounded (by a constant independent of $P$) and
\begin{align*}
\sum_{v\in S}\sum_{i=1}^q c_{i,v}\epsilon_{\tilde Y_{i,v}}(A) \lambda_{\tilde Y_{i,v},v}(P)\leq \sum_{v\in S}\sum_{j=1}^{m_v}c_{i_j,v}\epsilon_{\tilde Y_{i_{j,v},v}}(A) \lambda_{\tilde Y_{I_{j,v},v},v}(P). 
\end{align*}

Let $b_{j,v}=\codim \tilde Y_{I_{j,v},v}$ and set $b_{0,v}=0$.  Let $\epsilon_v(A)=\max_i \epsilon_{\tilde Y_{i,v}}(A)$.  We can therefore apply Corollary \ref{cor_elemineq} (for each $v\in S$) with
\begin{align}
a_1&= \epsilon_v(A) \lambda_{\tilde Y_{I_{1,v},v},v}(P) \geq \cdots \geq a_{m_v}= \epsilon_v(A) \lambda_{\tilde Y_{I_{m_v,v},v},v}(P)\geq 0, \label{chain_ineq_ai}  \\
b_1&=b_{1,v}-b_{0,v},\ldots, b_{m_v}=b_{m_v,v}-b_{m_v-1,v}, \nonumber \\
c_j&=c_{i_j,v}. \nonumber
\end{align}
to obtain, using the telescoping property for the $b_i$,
\begin{multline}\label{cheb_application}
\sum_{v\in S}  \left(\max_{j} \frac{\sum_{j'=1} ^{j} c_{i_{j'},v}}{b_{j,v}}\right) \sum_{s=1}^{m_v}(b_{s,v}-b_{s-1,v})\epsilon_v(A) \lambda_{\tilde Y_{I_{s,v},v},v}(P)\\
\geq\sum_{v\in S}\sum_{j=1}^{m_v}c_{i_j,v}\epsilon_v(A) \lambda_{\tilde Y_{I_{j,v},v},v}(P).
\end{multline}
According to \cite[Example 5.4.11]{PAGI} (which is phrased in terms of the $s$-invariant, which is the reciprocal of the Seshadri constant), the Seshadri constant of an intersection is bounded below by the minimum of the Seshadri constants of the factors of the intersection. Therefore, for any $j\in\{1,\ldots,m_v\}$, we have the inequality
\begin{equation}
\epsilon_{\tilde Y_{I_{j,v},v}}(A)\geq \min\{\epsilon_{\tilde Y_{i_1,v}}(A),\ldots,\epsilon_{\tilde Y_{i_j,v}}(A)\}\geq \epsilon_v(A)-\max_k \tau_{k,v} \cdot 2  \min_{l}  \epsilon_{\tilde  Y_{l,v}}(A)
,\label{laz_ineq}
\end{equation}
where the second inequality is due to \eqref{sesh_bound}. Then \eqref{cheb_application} implies
\begin{multline*}
\sum_{v\in S}  \left(\max_{j} \frac{\sum_{j'=1} ^{j} c_{i_{j'},v}}{b_{j,v}}\right) \sum_{s=1}^{m_v}(b_{s,v}-b_{s-1,v})  (\epsilon_{\tilde Y_{I_{s,v},v}}(A)  \\
 \ +\max_k \tau_{k,v} \cdot 2  \min_{l}  \epsilon_{\tilde  Y_{l,v}}(A))\lambda_{\tilde Y_{I_{s,v},v},v}(P)\\
 \geq  \sum_{v\in S}\sum_{j=1}^{m_v}c_{i_j,v}\epsilon_v(A) \lambda_{\tilde Y_{I_{j,v},v},v}(P) \geq  \sum_{v\in S}\sum_{j=1}^{m_v}c_{i_j,v}\epsilon_{\tilde Y_{i_{j,v},v}}(A) \lambda_{\tilde Y_{I_{j,v},v},v}(P).
\end{multline*}
We now note that if we form a new list with the closed subscheme $\tilde Y_{I_{s,v},v}$ repeated 
$b_{s,v}-b_{s-1,v}$ times (and omitted if $b_{s,v}-b_{s-1,v}=0$), then the resulting closed subschemes are in general position. By Theorem \ref{ajm_main}, there exists a proper Zariski closed subset $Z$ of $X$ such that
\begin{align*}
\sum_{v\in S} \sum_{s=1}^{m_v}(b_{s,v}-b_{s-1,v})\epsilon_{\tilde Y_{I_{s,v},v}}(A) \lambda_{\tilde Y_{I_{s,v},v},v}(P)<(n+1+\epsilon')h_A(P)
\end{align*}
for all points $P\in X(k)\setminus Z$. We also observe that by \eqref{modif_sesh} and additivity for local heights, for $i=1,\ldots,q$, we have
\begin{equation*}
\epsilon_{\tilde Y_{i,v}}(A) \lambda_{\tilde Y_{i,v},v}(P)=\frac 1 {c_v(1+\tau_{i,v})} \cdot c_v(1+\tau_{i,v}) \epsilon_{Y_{i,v}}(A)\lambda_{Y_{i,v},v}(P)= \epsilon_{Y_{i,v}}(A)\lambda_{Y_{i,v},v}(P).
\end{equation*}
For given $\epsilon''>0$, by choosing the quantities $\tau_{k,v}$ sufficiently small (depending only on data involving the $Y_{i,v}$, including the finite number of pertaining domination inequalities of the form  \eqref{lhf_ample_domination} for the $Y_{i,v}$), we have
\begin{align*}
& \sum_{v\in S}  \sum_{s=1}^{m_v}(b_{s,v}-b_{s-1,v})(\max_k \tau_{k,v} \cdot 2  \min_{l}  \epsilon_{\tilde  Y_{l,v}}(A)) \lambda_{\tilde Y_{I_{s,v},v},v}(P)\\
\leq &\ \sum_{v\in S}  \sum_{s=1}^{m_v}(b_{s,v}-b_{s-1,v})(\max_k \tau_{k,v} \cdot 2 \epsilon_{\tilde  Y_{i_{s,v},v}}(A)) \lambda_{\tilde Y_{i_{s,v},v},v}(P)\\
\leq &\ \sum_{v\in S}  \sum_{s=1}^{m_v}(b_{s,v}-b_{s-1,v})(\max_k \tau_{k,v} \cdot 2 \epsilon_{Y_{i_{s,v},v}}(A)) \lambda_{Y_{i_{s,v},v},v}(P)\\
<&\ \epsilon''h_A(P).
\end{align*}
Taking $W=\Supp \tilde Y_{I_{j,v},v}\neq \emptyset$, we also find
\begin{align*}
\frac{\sum_{j'=1} ^{j} c_{i_{j'},v}}{b_{j,v}}\leq \frac{\alpha_v(W)}{\codim W}.
\end{align*}
Note that although the indices $i_{1,v},\ldots, i_{m_v,v}, v\in S$, depended on the point $P$, there are only finitely many possibilities for these indices (and so only finitely many applications of Theorem \ref{ajm_main} are needed).   Putting everything together, we find that there exists a proper Zariski closed set $Z$ of $X$ such that
\begin{align*}
\sum_{v\in S}\sum_{i=1}^q c_{i,v}\epsilon_{Y_{i,v}}(A) \lambda_{Y_{i,v},v}(P)\leq \left((n+1)\max_{\substack{v\in S\\ \emptyset\subsetneq W\subsetneq X}}\frac{\alpha_v(W)}{\codim W}+\epsilon\right)h_A(P)
\end{align*}
for all points $P\in X(k)\setminus Z$, as desired.
\end{proof}

\section{Results of Quang and relations to other work on Schmidt's Subspace Theorem for divisors in subgeneral position} \label{quang_section}

Quang \cite{Quang} proved the following generalization of Evertse-Ferretti's theorem for divisors in subgeneral position:

\begin{theorem}[Quang]
Let $X$ be a projective variety of dimension $n$ defined over a number field $k$.  Let $S$ be a finite set of places of $k$.  Let $D_1,\ldots, D_q$ be effective Cartier divisors on $X$, defined over $k$, in $m$-subgeneral position.  Suppose that there exists an ample Cartier divisor $A$ on $X$ and positive integers $d_i$ such that $D_i\sim d_iA$ for all $i$.  Let $\epsilon>0$.  Then there exists a proper Zariski-closed subset $Z\subset X$ such that for all points $P\in X(k)\setminus Z$,
\begin{equation*}
\sum_{i=1}^q \frac{m_{D_i,S}(P)}{d_i}< ((m-n+1)(n+1)+\epsilon)h_A(P).
\end{equation*}
\end{theorem}

This improved an earlier result of the second author \cite{levin_duke} which had a factor of $\frac{m(m-1)(n+1)}{m+n-2}+\epsilon$ on the right-hand side.

Quang's theorem was refined by Ji, Yan, and Yu in \cite{JYY} using the notion of the index:

\begin{definition}\label{index_def}
If $X$ is a projective variety of dimension $n$, we say that closed subschemes $Y_1,\ldots, Y_q$ of $X$ are in {\it $m$-subgeneral position with index $\kappa$} if $Y_1,\ldots, Y_q$ are in $m$-subgeneral position and for every subset $J\subset\{1,\ldots, q\}$ with $|J|\leq \kappa$ we have $\codim \cap_{j\in J} Y_j\geq |J|$.
\end{definition}

Then Ji, Yan, and Yu proved:

\begin{theorem}[Ji, Yan, Yu]
\label{tJYY}
Let $X$ be a projective variety of dimension $n$ defined over a number field $k$.  Let $S$ be a finite set of places of $k$.  Let $D_1,\ldots, D_q$ be effective Cartier divisors on $X$, defined over $k$, in $m$-subgeneral position with index $\kappa$.  Suppose that there exists an ample Cartier divisor $A$ on $X$ and positive integers $d_i$ such that $D_i\sim d_iA$ for all $i$.  Let $\epsilon>0$.  Then there exists a proper Zariski-closed subset $Z\subset X$ such that for all points $P\in X(k)\setminus Z$,
\begin{equation*}
\sum_{i=1}^q \frac{m_{D_i,S}(P)}{d_i}< \left(\left(\frac{m-n}{\max\{1,\min\{m-n,\kappa\}\}}+1\right)(n+1)+\epsilon\right)h_A(P).
\end{equation*}
\end{theorem}

\begin{remark}
Ji, Yan, and Yu also claim that one can prove an inequality with $\max\{\frac{m}{2\kappa},1\}(n+1)$ on the right-hand side (\cite[Theorems 1.1 and 5.1]{JYY}). As they note, when $m\leq 2\kappa$ this gives a (claimed) coefficient of $n+1+\epsilon$ on the right-hand side. As observed already by Quang in \cite{Quang_Pacific}, their proof of this claim is incorrect, and we take the opportunity to give an explicit counterexample. Fix $k$ and $S$ with $|S|>1$. Take, for instance, $5$ distinct lines $L_1,\ldots, L_5$ in the plane with $4$ of them passing through a point $Q$ and let $D=L_1+\cdots +L_5$. Then $m=4$ and $\kappa=2$. Any line $L$ passing through $Q$ intersects the $5$ lines in $2$ points and so contains an infinite set $R_L$ of $(D,S)$-integral points. Then for such a line $L$ and points $P\in R_L\subset  L(k)$,
\begin{equation*}
\sum_{i=1}^5 m_{L_i,S}(P)=5h(P).
\end{equation*}
Since such lines $L$ are Zariski dense in $\mathbb{P}^2$ the best coefficient one can have in this case is $5$ (and not $3$ as given in their claim). More generally, it is easy to see that if one has at least $m+1$ hyperplanes in $\mathbb{P}^n$ in $m$-subgeneral position (and also not in $(m-1)$-subgeneral position) one cannot possibly do better than a coefficient of $m+1$ on the right-hand side, regardless of the index $\kappa$. \par
\end{remark}

Shi \cite{Shi} improved Theorem \ref{tJYY} in the Nevanlinna theory setting. The analogue of Shi's result in Diophantine approximation replaces the quantity $\left(\frac{m-n}{\max\{1,\min\{m-n,\kappa\}\}}+1\right)(n+1)$ in Theorem \ref{tJYY} by $\left(\frac{m-n}{\kappa}+1\right)(n+1)$.

More recently, Quang independently introduced the notion of a distributive constant $\Delta$ for a family of divisors in \cite{Quang_Pacific}, and in the recent preprint \cite{Quang22} Quang extended this notion to closed subschemes\footnote{There is a typo in \cite{Quang22}; a union should be replaced by an intersection in defining $\delta_{\mathcal{Y},X}$}:

\begin{definition}
The distributive constant of a family of closed subschemes $\mathcal{Y}=\{Y_1,\ldots, Y_q\}$ on a projective variety $X$ is given by
\begin{align*}
\delta_{\mathcal{Y}}=\max_{J\subset \{1,\ldots, q\}}\max\left\{1, \frac{\#J}{\codim \cap_{j\in J}\Supp Y_j}\right\},
\end{align*}
where we set $\codim\emptyset=\infty$.
\end{definition}

Then in \cite{Quang22}, Quang proves:

\begin{theorem}[Quang \cite{Quang22}]
\label{Quangmain}
Let $X$ be a  projective variety of dimension $n$  defined over a number field $k$, and let $S$ be a finite set of places of $k$. For each $v\in S$, let $Y_{1,v},\ldots, Y_{q,v}$ be closed subschemes of $X$, defined over $k$. Let $\Delta=\max_{v\in S}\delta_{\mathcal{Y}_v}$ be the maximum of the distributive constants of $\mathcal{Y}_v=\{Y_{1,v},\ldots, Y_{q,v}\}$. Let $A$ be an ample Cartier divisor on $X$, and $\epsilon>0$.  Then there exists a proper Zariski closed set $Z$ of $X$ such that
\begin{align*}
\sum_{v\in S}\sum_{i=1}^q \epsilon_{Y_{i,v}}(A) \lambda_{Y_{i,v},v}(P)<  \left(\Delta(n+1)+\epsilon\right)h_A(P)
\end{align*}
for all points $P\in X(k)\setminus Z$.
\end{theorem}

As noted in \cite{Quang_Pacific}, if the $Y_{i,v}$ are taken to be effective divisors in $m$-subgeneral position, then 
\begin{align*}
\Delta\leq m-n+1
\end{align*}
and if  the $Y_{i,v}$ are divisors in $m$-subgeneral position with index $\kappa$, then
\begin{align*}
\Delta\leq\frac{m-n}{\kappa}+1.
\end{align*}

Thus, Theorem \ref{Quangmain} contains the previous results of Quang, Ji-Yan-Yu, and Shi. If we set $c_{i,v}=1$ for all $i$ and $v$, then in comparison with Theorem \ref{main},
\begin{align*}
\delta_{\mathcal{Y}_v}=\max\left\{1, \max_{\emptyset\subsetneq W\subsetneq X}\frac{\alpha_v(W)}{\codim W}\right\}.
\end{align*}
This follows from the definitions and the easy observation that in the maximum over $W$, we may restrict to the case where $W$ is the intersection of a subset of closed subschemes of $\mathcal{Y}_v$.

Thus, we may view Theorem \ref{main} as a weighted version of Quang's Theorem \ref{Quangmain}, with an improvement in the coefficient of the height in Theorem \ref{main} in certain cases (due to the maximum with $1$ in Quang's definition of $\delta_{\mathcal{Y}})$.

In particular, Theorem \ref{main} gives weighted generalizations of the previous results of Quang \cite{Quang}, Ji-Yan-Yu \cite{JYY}, and Shi \cite{Shi}, in the general context of closed subschemes, using weighted generalizations of the notions of $m$-subgeneral position (see Section \ref{Nochkalow}) and the index. 

Note that Quang's proof of Theorem \ref{Quangmain} is an adaptation of our proof of Theorem \ref{AJM_mthm}, while our proof of Theorem \ref{main} is derived as a formal consequence of the statement of Theorem \ref{AJM_mthm} using properties of heights, Seshadri constants, and the generalized Chebyshev inequality of Lemma \ref{elemineq}.

\section{A General inequality of Nochka-type}\label{nochka_section}

This section is influenced by the approach of Vojta in \cite{vojta_nochka} to the seminal work of Nochka on weights \cite{Noc} (see Remark \ref{NDiagram1}). Let $X$ be a  projective variety of dimension $n$ and let $D_1,\ldots, D_q$ be effective Cartier divisors on $X$,  everything defined over a number field $k$. We first restrict to the classical case where $c_i=1$ for $i=1,\ldots, q$. Then for a closed subset $W\subset X$ we let
\begin{align*}
\alpha(W)=\#\{i\mid W\subset \Supp D_i\}.
\end{align*}

We prove the following easy corollary to Theorem \ref{main}. 

\begin{corollary}
\label{W0cor}
Let $W_0$ be a closed subset of $X$.  Let $\epsilon>0$.  Then there exists a proper Zariski closed set $Z$ of $X$ such that
\begin{align*}
\sum_{i=1}^q \epsilon_{D_i}(A) m_{D_i,S}(P)< \left(\alpha(W_0)+(n+1)\max_{\emptyset\subsetneq W\subsetneq X}\frac{\alpha(W)-\alpha(W\cup W_0)}{\codim W}+\epsilon\right)h_A(P)
\end{align*}
for all points $P\in X(k)\setminus Z$.
\end{corollary}

\begin{proof}
After reindexing, suppose that $W_0\subset \Supp D_j$ for the $\alpha(W_0)$ values of $j$ satisfying $q-\alpha(W_0)+1\leq j\leq q$. Let $q'=q-\alpha(W_0)$.  Let  
\begin{align*}
\alpha'(W)=\#\{i\leq q'\mid W\subset \Supp D_i\}.
\end{align*}
Then clearly $\alpha'(W)=\alpha(W)-\alpha(W\cup W_0)$. Using Lemma \ref{lemma}, we have the inequality
\begin{align*}
\sum_{i=q'+1}^q \epsilon_{D_i}(A) m_{D_i,S}(P)\leq (\alpha(W_0)+\epsilon)h_A(P)
\end{align*}
for all $P\in X(k)\setminus \cup_{i=q'+1}^q D_i$.

Then using the above and Theorem \ref{main}, there exists a proper Zariski closed subset $Z$ of $X$ such that
\begin{align*}
\sum_{i=1}^q \epsilon_{D_i}(A) m_{D_i,S}(P)&=\sum_{i=1}^{q'} \epsilon_{D_i}(A) m_{D_i,S}(P)+
\sum_{i=q'+1}^q \epsilon_{D_i}(A) m_{D_i,S}(P)\\
&\leq \left((n+1)\max_{\emptyset\subsetneq W\subsetneq X}\frac{\alpha'(W)}{\codim W}+\epsilon\right)h_A(P)+(\alpha(W_0)+\epsilon)h_A(P)\\
&\leq \left(\alpha(W_0)+(n+1)\max_{\emptyset\subsetneq W\subsetneq X}\frac{\alpha(W)-\alpha(W\cup W_0)}{\codim W}+2\epsilon\right)h_A(P)
\end{align*}
for all points $P\in X(k)\setminus Z$.
\end{proof}

Building on the method of proof of \cite[Main Theorem]{vojta_nochka}, we now prove Theorem \ref{nochka_type_theorem}, which we restate for convenience. Recall that $D_I=\cap_{i\in I}D_i$. 

\begin{theorem}
Let $X$ be a  projective variety of dimension $n$  defined over a number field $k$ and let $S$ be a finite set of places of $k$.  Let $D_1,\ldots, D_q$ be effective Cartier divisors on $X$, defined over $k$, in $m$-subgeneral position.  We assume the following Bezout property holds for intersections among the divisors: If $I,J\subset \{1,\ldots, q\}$ then
\begin{align*}
\codim D_{I\cup J}=\codim (D_I\cap D_J)\leq \codim D_I+\codim D_J.
\end{align*}
Let $A$ be an ample divisor on $X$ and let $\epsilon>0$.  Then there exists a proper Zariski closed subset $Z$ of $X$ such that
\begin{align*}
\sum_{i=1}^q \epsilon_{D_{i}}(A) m_{D_{i},S}(P)< \left(\frac{3}{2}(2m-n+1)+\epsilon\right)h_A(P)
\end{align*}
for all $P\in X(k) \setminus Z$.
\end{theorem}

\begin{remark}
\label{NDiagram1}
We briefly describe the main idea of the proof. Vojta constructs the classical Nochka weights (for hyperplanes) as slopes of line segments connecting the lower convex hull of the set of points consisting of  $(\alpha(W), \codim W)$, $\emptyset\subsetneq W\subset X$, and the point $P=(2m-n+1,n+1)$.  When this lower convex hull consists of just $(0,0)$ and $P$, the theorem follows easily from Theorem \ref{main}. Otherwise, we focus on the last line segment associated to the lower convex hull in the ``Nochka diagram," connecting (for some $W_0$) the points $(\alpha(W_0),\codim W_0)$ and $P$, with slope $\sigma$. Then we (essentially) assign the weight $0$ to divisors $D_i$ with $W_0\subset \Supp D_i$, and the weight $\sigma$ to all other divisors $D_i$. This partial use of the Nochka diagram loses a factor of $\frac{3}{2}$ (compared to the case of hyperplanes), but works in more generality.
\end{remark}

\begin{proof}

Suppose first that for every $W\subset X$, $W\neq \emptyset$,
\begin{align*}
\codim W\geq \frac{n+1}{2m-n+1}\alpha(W).
\end{align*}
Then it is immediate from Theorem \ref{main} that \eqref{almostNochka} holds with the better coefficient $2m-n+1+\epsilon$ on the right-hand side.

Otherwise, we choose $W_0$ in Corollary \ref{W0cor} such that the quantity
\begin{align*}
\frac{n+1-\codim W}{2m-n+1-\alpha(W)}
\end{align*}
is maximized at $W=W_0$. Replacing $W_0$ by the intersection of the divisors containing $W_0$, we may assume that $W_0$ is an intersection $D_I$ for some $I\subset \{1,\ldots, q\}$.  Let
\begin{align*}
\sigma=\frac{n+1-\codim W_0}{2m-n+1-\alpha(W_0)}
\end{align*}
(when the divisors $D_i$ are hyperplanes, this agrees with the slope $\sigma=\tau$ between the  points $P_s$ and $P_{s+1}$ in the corresponding Nochka diagram \cite[p.~232]{vojta_nochka}).

Let $\emptyset\subsetneq W\subsetneq X$.  By Corollary \ref{W0cor} we need to show that
\begin{align*}
\alpha(W_0)+(n+1)\frac{\alpha(W)-\alpha(W\cup W_0)}{\codim W}\leq \frac{3}{2}(2m-n+1).
\end{align*}

By replacing $W$ by the intersection of the divisors containing $W$, it suffices to consider the case that $W=D_J$ for some nonempty $J\subset\{1,\ldots, q\}$ (the case $\alpha(W)=0$, $J=\emptyset$, follows easily from $m$-subgeneral position).

We first claim that
\begin{align*}
\frac{\alpha(W)-\alpha(W\cup W_0)}{\codim W}\leq \frac{1}{\sigma}.
\end{align*}

We consider two cases.

Case I:  $W\cap W_0=\emptyset$.

In this case, by our Bezout assumption,
\begin{align*}
\codim W+\codim W_0\geq n+1,
\end{align*}
and from $m$-subgeneral position we have
\begin{align*}
\codim W_0&\geq \alpha(W_0)+n-m,\\
\codim W&\geq \alpha(W)+n-m.
\end{align*}

Then
\begin{align*}
\frac{1}{\sigma}&=\frac{2m-n+1-\alpha(W_0)}{n+1-\codim W_0}\\
&\geq \frac{m+1-\codim W_0}{n+1-\codim W_0}\\
&=1+\frac{m-n}{n+1-\codim W_0}\\
&\geq 1+\frac{m-n}{\codim W}\\
&=\frac{\codim W+m-n}{\codim W}\\
&\geq \frac{\alpha(W)}{\codim W}\\
&\geq \frac{\alpha(W)-\alpha(W\cup W_0)}{\codim W}
\end{align*}
as desired.

Case II:  $W\cap W_0\neq\emptyset$.

We use the two inequalities:
\begin{align*}
\alpha(W\cup W_0)+\alpha(W\cap W_0)\geq \alpha(W)+\alpha(W_0)
\end{align*}
and (from our Bezout assumption)
\begin{align*}
\codim W\geq \codim(W\cap W_0)-\codim W_0.
\end{align*}

Then
\begin{align*}
\frac{\alpha(W)-\alpha(W\cup W_0)}{\codim W}\leq \frac{\alpha(W\cap W_0)-\alpha(W_0)}{\codim(W\cap W_0)-\codim W_0}\leq \frac{1}{\sigma},
\end{align*}
where the second inequality follows from the definition of $W_0$ and $\sigma$.

Then in either case
\begin{align*}
\alpha(W_0)+(n+1)\frac{\alpha(W)-\alpha(W\cup W_0)}{\codim W}\leq \alpha(W_0)+\frac{n+1}{\sigma}.
\end{align*}

Finally, we note that from our assumptions $P=(\alpha(W_0),\codim W_0)$ lies below the line $y=\frac{n+1}{2m-n+1}x$.   From $m$-subgeneral position, it also lies to the left of the line $y=x+n-m$. Therefore $P$ must lie below and to the left of the point $W=(\frac{2m-n+1}{2},\frac{n+1}{2})$, which is the intersection of the two given lines.  Thus,
\begin{align*}
\alpha(W_0)&<\frac{2m-n+1}{2},\\
\codim W_0&< \frac{n+1}{2}.
\end{align*}
Since $\sigma>\frac{n+1}{2m-n+1}$ by assumption, it follows that
\begin{align*}
\alpha(W_0)+\frac{n+1}{\sigma}<\frac{3}{2}(2m-n+1)
\end{align*}
as desired.
\end{proof}

\begin{remark}
From an examination of the proof, it is possible to slightly improve the inequality \eqref{almostNochka}. For instance, at the end of the proof, one could use the better estimate $\alpha(W_0)\leq \lfloor\frac{2m-n}{2}\rfloor$, and similarly slightly improve the used lower bound for $\sigma$. We state the inequality in the present form both for its simplicity and easy relation with Nochka-Ru-Wong's inequality, and since we anticipate that further refinements of these methods may lead to more significant improvements. 
\end{remark}

\section{Nochka-Ru-Wong for surfaces and threefolds}
\label{Nochkalow}
We first extend the notion of $m$-subgeneral position to closed subschemes $Y_1,\ldots, Y_q$ with nonnegative weights $c_1,\ldots, c_q$ (we will reserve $\omega_i$ to denote certain Nochka-type weights below). Following the introduction, we write 
\begin{align*}
\alpha(W)=\sum_{\substack{i\\ W\subset \Supp Y_i}}c_{i}.
\end{align*}

Then we say that $Y_1,\ldots, Y_q$ with weights $c_1,\ldots, c_q$ are in $m$-subgeneral position if for every closed subset $\emptyset \neq W\subset X$, we have
\begin{align*}
\codim W\geq \alpha(W)+n-m.
\end{align*}

We now prove a general (weighted) version of Nochka-Ru-Wong's inequality for closed subschemes, assuming that one can construct ``Nochka weights" that satisfy certain properties. We will subsequently apply this result with appropriately chosen weights in dimensions $\leq 3$.

\begin{theorem}\label{generalnochka}
Let $X$ be a projective variety of dimension $n$  defined over a number field $k$ and let $S$ be a finite set of places of $k$.  Let $Y_1,\ldots, Y_q$ be closed subschemes of $X$, defined over $k$, with nonnegative real weights $c_1,\ldots, c_q$.  Suppose that there exist nonnegative real weights $\omega_1,\ldots, \omega_q\geq 0$, not all $0$, such that for any nonempty closed subset $W\subset X$,
\begin{align}
\label{NE}
\alpha^{\text{Nochka}}(W):= \sum_{\substack{i\\ W\subset \Supp Y_i}}c_{i}\omega_i\leq \codim W.
\end{align}
Let $\tau=\max_i \omega_i$ and let
\begin{align*}
B=\frac{n+1}{\tau}+\sum_{i=1}^q c_i\left(1-\frac{\omega_i}{\tau}\right).
\end{align*}
Let $\epsilon>0$.  Then there exists a proper Zariski closed subset $Z$ of $X$ such that
\begin{align*}
\sum_{i=1}^q c_i\epsilon_{Y_{i}}(A) m_{Y_{i},S}(P)<(B+\epsilon)h_A(P)
\end{align*}
for all $P\in X(k) \setminus Z$.
\end{theorem}

\begin{proof}
We split the proximity functions as
\begin{align*}
\sum_{i=1}^q c_i\epsilon_{Y_{i}}(A) m_{Y_{i},S}(P)=\sum_{i=1}^q \left(c_i-\frac{c_i\omega_i}{\tau}\right)\epsilon_{Y_{i}}(A)m_{Y_i,S}(P)+\frac{1}{\tau}\sum_{i=1}^q c_i\omega_i\epsilon_{Y_{i}}(A)m_{Y_i,S}(P).
\end{align*}

We have the trivial bound coming from Lemma \ref{lemma2}:
\begin{align*}
\sum_{i=1}^q \left(c_i-\frac{c_i\omega_i}{\tau}\right)\epsilon_{Y_{i}}(A)m_{Y_i,S}(P)\leq \sum_{i=1}^q \left(c_i-\frac{c_i\omega_i}{\tau}+\epsilon\right)h_A(P).
\end{align*}
On the other hand, by Theorem \ref{main} applied with $\alpha^{\text{Nochka}}(W)$ from \eqref{NE}, there exists a proper Zariski closed subset $Z$ of $X$ such that
\begin{align*}
\frac{1}{\tau}\sum_{i=1}^q c_i\omega_i\epsilon_{Y_{i}}(A)m_{Y_i,S}(P)\leq \left(\frac{n+1}{\tau}+\epsilon\right)h_A(P)
\end{align*}
for all $P\in X(k) \setminus Z$.
Combining the above inequalities yields
\begin{align*}
\sum_{i=1}^q c_i\epsilon_{Y_{i}}(A) m_{Y_{i},S}(P)<(B+\epsilon)h_A(P)
\end{align*}
for all $P\in X(k) \setminus Z$, as desired.
\end{proof}

We now show that under mild hypotheses, one can find suitable ``Nochka weights" in dimensions $\leq 3$ and derive a generalization of Nochka-Ru-Wong's inequalities.
\begin{remark}
A key feature in dimensions $\leq 3$ is that the lower convex hull in the ``Nochka diagram" (see Remark \ref{NDiagram1}) yields at most $2$ line segments, which permits a finer analysis of the problem. 
\end{remark}

\begin{theorem}\label{nochka23irred}
Let $X$ be a projective variety of dimension $n\leq 3$  defined over a number field $k$ and let $S$ be a finite set of places of $k$.  Let $D_1,\ldots, D_q$ be ample effective Cartier divisors on $X$, defined over $k$, with nonnegative real weights $c_1,\ldots, c_q$, in $m$-subgeneral position.  Suppose that $\Supp D_i$ is irreducible for all $i$.  Let $\epsilon>0$.  Then there exists a proper Zariski closed subset $Z$ of $X$ such that
\begin{align*}
\sum_{i=1}^q c_i\epsilon_{D_{i}}(A) m_{D_{i},S}(P)<(2m-n+1+\epsilon)h_A(P)
\end{align*}
for all $P\in X(k) \setminus Z$.
\end{theorem}

\begin{proof}
When $n=1$ this is immediate from the definition of $m$-subgeneral position and Theorem \ref{main}. Suppose $n\in \{2,3\}$. Let
\begin{align*}
c=\max_{\substack{W\subset X\\\codim W=1}}\alpha(W)=\max_{\substack{W\subset X\\\codim W=1}}\frac{\alpha(W)}{\codim W}.
\end{align*}
First, we note that by Theorem \ref{main} the desired result holds if
\begin{align*}
\frac{\alpha(W)}{\codim W}\leq \frac{2m-n+1}{n+1}
\end{align*}
for all nonempty proper closed subsets $W\subset X$. By definition of $m$-subgeneral position, $\alpha(W)\leq \codim W+m-n$, and if $\codim W\geq \frac{n+1}{2}$, then
\begin{align*}
\frac{\alpha(W)}{\codim W}\leq 1+\frac{m-n}{\codim W}\leq 1+\frac{2(m-n)}{n+1}=\frac{2m-n+1}{n+1}.
\end{align*}
For $n\in \{2,3\}$, $\codim W\geq \frac{n+1}{2}$ if and only if $\codim W>1$. Thus, the result follows from Theorem \ref{main} unless 
\begin{align}
\label{ccond}
c>\frac{2m-n+1}{n+1},
\end{align}
which we now assume.

Let $W_0$ be an irreducible closed subset of $X$ with $\codim W_0=1$ such that $c=\alpha(W_0)$. We define
\begin{align*}
\omega_i=
\begin{cases}
\frac{1}{c} & \text{if $\Supp D_i=W_0$}\\
\frac{n}{2m-n+1-c} & \text{otherwise}.
\end{cases}
\end{align*}
The condition \eqref{ccond} implies that $\frac{1}{c}<\frac{n}{2m-n+1-c}$ and so in particular $\omega_i\leq \frac{n}{2m-n+1-c}$ for all $i$.

We claim that for any nonempty closed subset $W\subset X$,
\begin{align*}
\sum_{\substack{i\\ W\subset \Supp D_i}}c_{i}\omega_i\leq \codim W.
\end{align*}
Suppose first that $\dim W=0$. From $m$-subgeneral position, $c\leq 1+m-n$ and $\alpha(W)\leq m$. Then
\begin{align*}
\sum_{\substack{i\\ W\subset \Supp D_i}}c_{i}\omega_i\leq \frac{n}{2m-n+1-c}\alpha(W)\leq \frac{n}{m}m\leq n=\codim W.
\end{align*}
Suppose now that $\codim W=1$. We may assume $W$ is irreducible and $W=D_j$ for some $j$. If $W=W_0$, then from the definitions we have
\begin{align*}
\sum_{\substack{i\\ W\subset \Supp D_i}}c_{i}\omega_i=\frac{1}{c}\alpha(W_0)=1.
\end{align*}
Otherwise, since the divisors $D_i$ are ample, we have $\codim (W\cap W_0)=2$ and
\begin{align*}
\alpha(W)+\alpha(W_0)=\alpha(W)+c&\leq \alpha(W\cap W_0)\\
&\leq \codim(W\cap W_0)+m-n=m-n+2.
\end{align*}
Then
\begin{align*}
\sum_{\substack{i\\ W\subset \Supp D_i}}c_{i}\omega_i=\frac{n}{2m-n+1-c} \alpha(W)\leq \frac{n(m-n+2-c)}{2m-n+1-c}.
\end{align*}
Suppose first that $n=2$. Then from $m$-subgeneral position with $W=X$, we have $m\geq n\geq 2$. Then from \eqref{ccond}, $c\geq \frac{2m-1}{3}\geq 1$, and this implies
\begin{align*}
\frac{n(m-n+2-c)}{2m-n+1-c}=\frac{2(m-c)}{2m-1-c}\leq 1=\codim W
\end{align*}
as desired. If $n=3$, by \eqref{ccond}, $c>\frac{2m-2}{4}=\frac{m-1}{2}$, and this implies
\begin{align}
\label{n3calc}
\frac{n(m-n+2-c)}{2m-n+1-c}=\frac{3(m-1-c)}{2m-2-c}\leq 1=\codim W.
\end{align}
Finally, we consider the case $n=3$ and $\dim W=1$, $\codim W=2$. Suppose first that $W\subset W_0$. Let $\alpha'=\alpha(W)-\alpha(W_0)=\alpha(W)-c$. Then
\begin{align*}
\alpha(W)=\alpha'+c\leq \codim(W)+m-n=m-1.
\end{align*}
From the definitions and the calculation in \eqref{n3calc}, we find
\begin{align*}
\sum_{\substack{i\\ W\subset \Supp D_i}}c_{i}\omega_i&=\alpha'\frac{n}{2m-n+1-c}+c\frac{1}{c}\\
&\leq \frac{3(m-1-c)}{2m-2-c}+1\leq 2=\codim W.
\end{align*}
Suppose now that $W\not\subset W_0$. Then from ampleness, $\dim (W\cap W_0)=0$ and from $m$-subgeneral position,
\begin{align*}
\alpha(W)+\alpha(W_0)=\alpha(W)+c&\leq \alpha(W\cap W_0)\\
&\leq \codim(W\cap W_0)+m-n=m.
\end{align*}
Therefore, since $m\geq n\geq 3$ and $c>\frac{m-1}{2}\geq 1$, we easily find
\begin{align*}
\sum_{\substack{i\\ W\subset \Supp D_i}}c_{i}\omega_i&=\alpha(W)\frac{n}{2m-n+1-c}\leq \frac{3(m-c)}{2m-2-c}\\
&\leq 2=\codim W.
\end{align*}
This proves \eqref{NE} of Theorem \ref{generalnochka}. Finally, we have
\begin{align*}
\tau=\max\omega_i=\frac{n}{2m-n+1-c}
\end{align*}
and we compute 
\begin{align*}
B&=\frac{n+1}{\tau}+\sum_{i=1}^q c_i\left(1-\frac{\omega_i}{\tau}\right)=\frac{n+1}{\tau}+\sum_{\substack{i\\\Supp D_i=W_0}}c_i\left(1-\frac{1}{c\tau}\right)\\
&=\frac{n+1}{\tau}+c\left(1-\frac{1}{c\tau}\right)=\frac{n+1}{\tau}+c-\frac{1}{\tau}=\frac{n}{\tau}+c\\
&=(2m-n+1-c)+c=2m-n+1.
\end{align*}
Then an application of Theorem \ref{generalnochka} completes the proof. \end{proof}

Finally, under some additional hypotheses, we can remove the condition that $\Supp D_i$ is irreducible. In particular, Theorem \ref{nochka_type_theorem2} follows from:

\begin{theorem}
\label{theorempic1}
Let $X$ be a nonsingular projective variety of dimension $n\leq 3$,  defined over a number field $k$, with Picard number $\rho=1$. Let $S$ be a finite set of places of $k$.  
Let $D_1,\ldots, D_q$ be effective divisors on $X$, defined over $k$, with nonnegative weights $c_1,\ldots, c_q$, in $m$-subgeneral position.  Let $A$ be an ample divisor on $X$ and let $\epsilon>0$.  Then there exists a proper Zariski closed subset $Z$ of $X$ such that
\begin{align*}
\sum_{i=1}^q c_i\epsilon_{D_{i}}(A) m_{D_{i},S}(P)<(2m-n+1+\epsilon)h_A(P)
\end{align*}
for all $P\in X(k) \setminus Z$.
\end{theorem}

\begin{proof}
For $i=1,\ldots, q$, write 
\begin{align*}
D_i=\sum_{j=1}^{r_i}D_{ij}
\end{align*}
where the $D_{ij}$ are (not necessarily distinct) effective divisors with $\Supp D_{ij}$ irreducible. We have
\begin{align*}
m_{D_{i},S}(P)=\sum_{j=1}^{r_i}m_{D_{ij},S}(P).
\end{align*}
Since $\rho=1$, we have $D_{ij}\equiv a_{ij}A$, for some $a_{ij}> 0$, and  $\epsilon_{D_{ij}}(A)=\frac{1}{a_{ij}}$. Similarly, $D_i\equiv \sum_{j=1}^{r_i} a_{ij}A$, and so
\begin{align*}
\epsilon_{D_i}(A)=\frac{1}{\sum_{i=1}^{r_i}a_{ij}}.
\end{align*}

Then we find
\begin{align*}
\sum_{j=1}^{r_i}\frac{\epsilon_{D_i}(A)}{\epsilon_{D_{ij}}(A)}=1.
\end{align*}

Now we write
\begin{align}
\label{redtoirred}
\sum_{i=1}^q c_i\epsilon_{D_{i}}(A) m_{D_{i},S}(P)=\sum_{i=1}^q \sum_{j=1}^{r_i}\left(c_i\frac{\epsilon_{D_{i}}(A)}{\epsilon_{D_{ij}}(A)}\right)\epsilon_{D_{ij}}(A)D_{ij}.
\end{align}

We claim that the divisors $D_{ij}$ with weights $c_i\frac{\epsilon_{D_{i}}(A)}{\epsilon_{D_{ij}}(A)}$, $i=1,\ldots, q$, $j=1,\ldots, r_i$, are in $m$-subgeneral position. Let $W\subset X$ be a nonempty closed subset. Then
\begin{align*}
\sum_{\substack{i,j\\ W\subset \Supp D_{ij}}}c_i\frac{\epsilon_{D_{i}}(A)}{\epsilon_{D_{ij}}(A)} \leq \sum_{\substack{i\\ W\subset \Supp D_{i}}}c_i\sum_{j=1}^{r_i}\frac{\epsilon_{D_{i}}(A)}{\epsilon_{D_{ij}}(A)} \leq \sum_{\substack{i\\ W\subset \Supp D_{i}}}c_i.
\end{align*}
Since the divisors $D_1,\ldots, D_q$, with weights $c_1,\ldots, c_q$, are in $m$-subgeneral position, the claim follows. We complete the proof by using  \eqref{redtoirred} and Theorem \ref{nochka23irred}, applied to the ample divisors $D_{ij}$ (with appropriate weights).
\end{proof}

\end{document}